\documentclass[11pt]{article}
\usepackage{amsmath,amsfonts,amsthm,amssymb}
\usepackage{enumerate}
\usepackage{parskip}
\usepackage{fullpage}
\usepackage[breakable,skins]{tcolorbox}
\usepackage{hyperref}
\usepackage[
backend=biber,
style=alphabetic,
maxnames=10
]{biblatex}
\bibliography{paper.bib}

\newcommand{\abs}[1]{\left\lvert#1\right\rvert}
\newcommand{\norm}[1]{\left\lVert#1\right\rVert}

\makeatletter
\def\blfootnote{\xdef\@thefnmark{}\@footnotetext}
\makeatother

\newtheorem{thm}{Theorem}[section]
\newtheorem{cor}{Corollary}[section]
\newtheorem{lem}{Lemma}[section]

\theoremstyle{remark}
\newtheorem{rem}{Remark}[section]

\theoremstyle{definition}
\newtheorem{defn}{Definition}[section]

\newcommand{\C}{{\mathbb C}}

\newcommand{\N}{{\mathbb N}}

\newcommand{\R}{{\mathbb R}}

\title{Establishing strong 1-boundedness via non-microstates free entropy techniques}
\author{Benjamin Major, Dimitri Shlyakhtenko}
\date{}

\begin{document}

\maketitle
\begin{abstract}
    We show that, for many choices of finite tuples of generators $\mathbf{X}=(x_1,\dots,x_d)$ of a tracial von Neumann algebra $(M,\tau)$ satisfying certain decomposition properties (non-primeness, possessing a Cartan subalgebra, or property $\Gamma$), one can find a diffuse, hyperfinite subalgebra $N\subseteq W^*(\mathbf{X})^\omega$ (often in $W^*(\mathbf{X})$ itself), such that $$W^*(N,\mathbf{X}+\sqrt{t}\mathbf{S})=W^*(N,\mathbf{X},\mathbf{S})$$ for all $t>0$. (Here $\mathbf{S}$ is a free semicircular family, free from $\{\mathbf{X}\cup N\}$). This gives a short 'non-microstates' proof of strong 1-boundedness for such algebras.
\end{abstract}
\section{Introduction}

Voiculescu's free entropy dimension \cite{Voi96} is defined for $n$-tuples of self-adjoint elements in a tracial von Neumann algebra and is related to the short-term behavior of free entropy under semicircular perturbations. Since the free entropy dimension of a generating set of a diffuse amenable von Neumann algebra is always $1$ (cf. \cite{JungHyperfinite}), it makes sense to consider a kind of relative quantity (see \cite{Shl02,Jun07,JP24}). A strengthened version of the vanishing of such a relative quantity is called strong 1-boundedness and was defined 
by Jung \cite{Jun07} for finite tuples of self-adjoint elements in a tracial von Neumann algebra (see also \cite{Ge96,HS07,Hay18}). 
Being strongly 1-bounded is an invariant of the von Neumann algebra that the tuple generates, and is implied by that algebra possessing a Cartan subalgebra, property $\Gamma$, tensor product decomposition, and most instances of Property (T) (\cite{Voi96,Jun07,JS07,Hay18,HJKE25}).  On the other hand, Connes-embeddable von Neumann algebras that are free products (e.g. free group factors) are not strongly $1$-bounded. 

Voiculescu also introduced  a non-microstates free entropy $\chi^*$ which is defined via the free Fisher information $\Phi^*$ (\cite{Voi98-1}). There are conditions on $\chi^*$ similar to strong 1-boundedness, which can be stated in terms of $\Phi^*$ (see e.g. \cite{Shl21} for some applications to von Neumann algebras of groups).  In this paper we will be considering a version of these conditions relative to a diffuse abelian sub-algebra. 

Suppose  that $\mathbf{X}$ is a self-adjoint $d$-tuple in a tracial von Neumann algebra, $N$ is some fixed von Neumann subalgebra, and $\mathbf{S}$ is a free $(0,1)$-semicircular $d$-tuple, free from $N\cup\{\mathbf{X}\}$. Let $\Phi^*(\cdots :N)$ denote Voiculescu's free Fisher relative to $N$ (\cite{Voi98-1}).  Then 
\begin{equation}\label{eq1}
    \Phi^*(\mathbf{X}+\sqrt{\epsilon}\mathbf{S}:N)\leq\frac{d}{\epsilon}\quad\forall\,\epsilon>0.
\end{equation} with equality iff $N\vee W^*(\mathbf{X}+\sqrt{\epsilon}\mathbf{S})=N\vee W^*(\mathbf{X},\mathbf{S})$,  i.e., iff $\mathbf{X}\subset N\vee W^*(\mathbf{X}+\sqrt{\epsilon}\mathbf{S})$.

Thus if e.g. $\mathbf{X}=\mathbf{0}$ (or more generally $\mathbf{X}\subset N$), equality in (\ref{eq1}) holds trivially for all $\epsilon>0$. 

Our main result is that this phenomenon also occurs when $\mathbf{X}$ is one of many choices of generators for the key examples of strongly 1-bounded algebras. We show that in these cases one can find a diffuse abelian subalgebra $N$ of an ultrapower $W^*(\mathbf{X})^\omega$ (an often even of $W^*(\mathbf{X})$ itself) such that $\mathbf{X}\subset N\vee W^*(\mathbf{X}+\sqrt{\epsilon}\mathbf{S})$ for all $\epsilon>0$. The key ingredient is the consideration of $N\vee W^*(\mathbf{X}+\sqrt{\epsilon}\mathbf{S})$ as an $N,N$ bimodule, and the identification of $\mathbf{X}$ as the ``non-coarse part'' of the vector associated to $\mathbf{X}+\sqrt{\epsilon}\mathbf{S}$.
We then show that such an estimate on the free Fisher information (we can loosen the conditions to require only that $N$ be diffuse and hyperfinite) implies the strong 1-boundedness of $W^*(\mathbf{X})$. Indeed, we estimate the relative non-microstates free entropy $\chi^*(\cdots :N )$, and appeal to the results of \cite{JP24} to establish strong enough bounds on the microstates free entropy $\chi$. This provides a quick proof that $W^*(\mathbf{X})$ is strongly 1-bounded, indeed has non-positive 1-bounded entropy. 

\textbf{Acknowledgments:}
This research was sponsored in part by the Army Research Office and was accomplished under Grant
Number W911NF-25-1-0075. The views and conclusions contained in this document are those of the authors and
should not be interpreted as representing the official policies, either expressed or implied, of the Army Research
Office or the U.S. Government. The U.S. Government is authorized to reproduce and distribute reprints for
Government purposes notwithstanding any copyright notation herein.
Research was also supported in part by NSF grant DMS-2348633.

\section{Main Results}
This section presents self-contained arguments for our main theorems, and we defer the free probabilistic preliminaries to the latter half of the paper.

Throughout, we fix $\mathbf{X}=(x_1,\dots,x_d)$ a self-adjoint tuple of random variables in a tracial von Neumann algebra $(M,\tau)$, and assume $W^*(\mathbf{X})=M$. Let further $\mathbf{S}=(s_1,\dots,s_d)$ be a free semicircular family, free from $W^*(\mathbf{X})$. Also, $N$ will always denote a diffuse hyperfinite von Neumann algebra, most often of $M$, but we will on occasion take it to be in $M^\omega$, for a free ultrafilter $\omega\in\beta\N\setminus\N$. We denote also $\widetilde{M}=W^*(\mathbf{X},\mathbf{S})$, and $M_t:=W^*(N,\mathbf{X}+\sqrt{t}\mathbf{S})\subseteq \widetilde{M}^\omega$ for $t\geq0$. It will be clear from context when $N$ is in the ultrapower as opposed to $M$ itself.

Denote by $N^\text{op}$ the opposite algebra of $N$, and by $N\odot N^{\text{op}}$ the algebraic tensor product of $N$ and $N^\text{op}$.  For $\beta=\sum_{i=1}^na_i\odot b_i^{\text{op}}\in N\odot N^{\text{op}}$ and $y\in M$, we write $\beta\#y$ to mean $\sum_{i=1}^n a_iyb_i$.
\begin{thm}
    Suppose that one can find a free ultrafilter $\omega$ on a countably infinite set and $N\subseteq W^*(\mathbf{X})^\omega$ diffuse, separable, and abelian, such that $L^2(N\mathbf{X}N)\perp L^2(N)\otimes L^2(N)$ as $N-N$ bimodules. Then \begin{equation} \label{maineq}
        M_t=N\vee W^*(\mathbf{X}+\sqrt{t}\mathbf{S})=W^*(N,\mathbf{X},\mathbf{S})=N\vee \widetilde{M}
    \end{equation} for all $t>0$.  Equivalently, 
    \begin{equation}\label{eq:maineq1}
        \mathbf{X}\subset W^*(N,\mathbf{X}+\sqrt{t}\mathbf{S}).
    \end{equation}
\end{thm}

\begin{proof}
Consider $N\vee \widetilde{M}$ with its canonical trace, which we also denote by $\tau$. There is an atomless standard probability space $(\mathcal{X},\mu)$ such that $(L^\infty(\mathcal{X},\mu),\int(\cdot)\,d\mu)\cong (W^*(\mathbf{X}),\tau)$ (here we use the separability assumption on $N$).

By a classical construction (see e.g. \cite{Voi96}, Section 7), each $\xi\in L^2(P,\tau)$ determines a measure $\nu_\xi$ on $\mathcal{X}$. For each $s_i$, we have $\nu_{s_i}=\mu\otimes\mu$, while the condition $L^2(N\mathbf{X}N)\perp L^2(N)\otimes L^2(N)$ is equivalent to $\nu_{x_i}\perp\mu\otimes\mu$ (i.e. the measures are mutually singular).

In this case, there is a net $\{\alpha_\lambda\}_{\lambda\in\Lambda}\in N\odot N^{\text{op}}$, where $\alpha_\lambda=\sum_{n=1}^{N(\lambda)}p^{(1)}_n\otimes p_n^{(2)}$, with each $p_n^{(i)}$ a projection, so that $\norm{\alpha_\lambda\# s_i-s_i}_2\to0$ and $\norm{\alpha_\lambda\# x_i}_2\to0$ for each $i$ (see again \cite{Voi96}, or \cite{Hay18}). Denote by $E_t:P\to M_t$ the trace-preserving conditional expectation. Recall that it is $N-N$ bimodular. So, we have $$\sqrt{t}s_i=\lim_\lambda\alpha_\lambda\#(x_i+\sqrt{t}s_i)=\lim_\lambda\alpha_\lambda\# E_t[x_i+\sqrt{t}s_i]=\lim_\lambda E_t[\alpha_\lambda\#(x_i+\sqrt{t}s_i)]=\sqrt{t}E_t(s_i).$$
Hence $E_t(s_i)=s_i$, $1\leq i\leq d$, and so $\mathbf{S}\subset M_t$.  Since $\mathbf{X}+\sqrt{t}\mathbf{S}\subset M_t$, we get \eqref{eq:maineq1}. Thus  $N\vee \widetilde{M} \supset M_t \supset N\vee \widetilde{M}$ giving us \eqref{maineq}.
\end{proof}
\begin{rem}
    The ultrapower played no role in the argument. We include it in the statement of the theorem to emphasize that its conclusion remains valid when $N$ is in an ultrapower.
\end{rem}

\begin{cor}
Let $(M,\tau)$ be a finitely-generated diffuse tracial von Neumann algebra. If (a) $M$
has a Cartan subalgebra, or
    if (b) $M$ has  property $\Gamma$ (in particular, if $M$ has diffuse center or is hyperfinite), then there is a diffuse abelian $N$ as in Theorem 2.1 such that (\ref{maineq}) holds for any finite set of generators $\mathbf{X}$ for $M$. 
\end{cor}
\begin{proof}

    a): Let $N\leq M$ be a Cartan subalgebra, so that $M=\mathcal{N}_M(N)''$. So, any s.a. tuple of generators generates a bimodule disjoint from $L^2(N)\otimes L^2(N)$, i.e. is singular with respect to $N$ in the sense of \cite{Voi93}, and this means the conditions for (\ref{maineq}) to hold are satisfied.
    
    b): In this case, there is an $\omega\in\beta\N\setminus\N$ and a diffuse abelian $N\leq M'\cap M^\omega$ by \cite{Dix69}. By taking an element $a\in N$ with diffuse spectrum and replacing $N$ by $W^*(a)$, we may assume $N$ is separable. It follows that $M\leq\mathcal{N}_{M^\omega}(N)''$ and we may conclude as before.  
\end{proof}

We now give some examples where (\ref{maineq}) is satisfied only for certain choices of generators.
\begin{cor}
\begin{enumerate}
    \item Suppose $M$ is non-prime, i.e. $M\cong M_1\overline{\otimes}M_2$ with $M_1,M_2$ diffuse, and fix $N_1\leq M_1,N_2\leq M_2$ diffuse abelian subalgebras. Then, (\ref{maineq}) holds when the generating set $\mathbf{X}$ is chosen from $M_1\overline{\otimes}N_2\cup N_1\overline\otimes{M_2}$.
    \item If $M$ is generated by unitaries $v_1,\dots,v_k$ which have Property $C'$ in the sense of \cite{GaPo15} (see proof for description), then one can find $N$ depending on the $v_i$ such that (\ref{maineq}) holds.
\end{enumerate}
\end{cor}
\begin{proof}
    First, suppose $M$ is non-prime; write $M\cong M_1\overline{\otimes}M_2$, and fix $N_1\leq M_1$, $N_2\leq M_2$ diffuse abelian algebras. Observe that $M_1\overline{\otimes} N_2\leq\mathcal{N}_M(\C1\overline{\otimes}N_2)''$, and similarly $N_1\overline{\otimes}M_2\leq\mathcal{N}_M(N_1\overline{\otimes}\C1)''$. Then $N=N_1\overline{\otimes}N_2$ works.

    Next, suppose $M=W^*(v_1,\dots,v_k)$ for unitaries $v_1,\dots,v_k$ with Property $C'$. This means that the $v_i$ have diffuse spectrum and there are mutually commuting unitaries $u_1,\dots,u_k\in M^\omega$ for some free ultrafilter $\omega$, also with diffuse spectrum, such that $[v_j,u_j]=0$ for $1\leq j\leq k$. Let $N=W^*(u_1,\dots,u_k)$, which is abelian. 
    
    Note, though, that $N$ may not be separable. However, each $W^*(u_j)$ is, and so we may repeat the argument of Theorem 2.1 to find, for each $1\leq j\leq k$, a net $\{\beta^{(j)}\}_{i\in I}$ in $W^*(u_j)\odot W^*(u_j)^{\text{op}}\subseteq N\odot N^{\text{op}}$ so that $\beta_i^{(j)}\# s_j\to s_j$ and $\beta_i^{(j)}\#x_j\to0$. 
\end{proof}
\subsection{Non-microstates free entropy}

Let $N\subseteq M$ be a unital $*$-subalgebra, $\mathbf{S}=(s_1,\dots,s_d)$ a free $(0,1)$-semicircular family, free from $N\cup\mathbf{X}$; denote by $E_\epsilon:M\to W^*(N,\mathbf{X+\sqrt{\epsilon}S})$ the trace-preserving conditional expectation. 

The \textit{free Fisher information relative to N}, $\Phi^*((\cdot):N)$, is defined in terms of the conjugate variables $\mathcal{J}(x_i:N)$. We will not need the precise definition of either quantity. Instead, we record that Corollary 3.9 in \cite{Voi98-1} shows $\mathcal{J}(x_i+\sqrt{\epsilon}s_i:N)=\frac{1}{\sqrt{\epsilon}}E_\epsilon(s_i)$, and so by the definition (see Definition 6.1 in \cite{Voi98-1}) of the free Fisher information, we have $$\Phi^*(\mathbf{X+\sqrt{\epsilon}S}: N)=\sum_{i=1}^N\norm{\mathcal{J}(x_i+\sqrt{\epsilon}s_i:N)}_2^2=\frac{1}{\epsilon}\sum_{1\leq i\leq d}\norm{E_\epsilon(s_i)}_2^2.$$ 

We then define the \textit{non-microstates free entropy of} \textbf{X} \textit{relative to $N$} by $$\chi^*(\mathbf{X}:N):=\frac{1}{2}\int_0^\infty\left[\frac{d}{1+t}-\Phi^*(\mathbf{X}+\sqrt{t}\mathbf{S}:N)\right]\,dt+\frac{d}{2}\log(2\pi e).$$

Using the Kaplansky density theorem, these formulae hold with $N$ replaced by $W^*(N)$ (\cite{Voi98-1}). We will thus without comment take $N$ to be a von Neumann subalgebra of $M$.

Now we record the implication (\ref{maineq}) has for non-microstates free entropy.
\begin{cor}
    If $\mathbf{X}=(x_1,\dots,x_d)$ and $N$ satisfy (\ref{maineq}), and $\mathbf{S}$ is a free semicircular family, free from $\mathbf{X}\cup N$, then $$\Phi^*(\mathbf{X+\sqrt{\epsilon}S}:N)=\frac{d}{\epsilon},\text{ and }\chi^*(\mathbf{X+\sqrt{\epsilon}S}:N)=\frac{d}{2}[\log(2\pi e)+\log\epsilon]$$ for all $\epsilon>0$. 
\end{cor}
\begin{proof}
    The first assertion follows from the formula given for the free Fisher information above and the fact that, in our case, $E_\epsilon=\text{id}\rvert_{W^*(N,\mathbf{X},\mathbf{S})}$ for all $\epsilon>0$. 
    
    For the statement about $\chi^*$, let $\mathbf{S,S',S''}$ be free semicircular families, free from each other as well as $\mathbf{X}\cup N$. Observe that $(\sqrt{\delta_1+\delta_2})s_j$ and $\sqrt{\delta_1}s_j'+\sqrt{\delta_2}s_j'$ have the same $*$-distribution; in particular, the latter are semicircular variables, free from $\mathbf{X}\cup N$. Then the above arguments easily show $$\Phi^*(\mathbf{X+(\sqrt{\delta_1+\delta_2})S}:N)=\Phi^*(\mathbf{X+\sqrt{\delta_1}S'+\sqrt{\delta_2}S''}:N),$$ and hence $$\chi^*(\mathbf{X+\sqrt{\epsilon}S}:N)=\frac{1}{2}\int_0^\infty\left[\frac{d}{1+t}-\frac{d}{\epsilon+t}\right]\,dt + \frac{d}{2}\log(2\pi e)=\frac{d}{2}\log(2\pi e\epsilon).$$ 
\end{proof}

\subsection{Conditional Microstates Free Entropy}
We briefly review (conditional) microstates free entropy. For microstates free entropy, we refer to \cite{Voi93,Voi94,Voi96}, while for the conditional version, to \cite{Shl02,JP24}. In this paper, we use the notation of \cite{JP24}. 

For $m,k\in\N,\gamma>0$, Voiculescu's microstate spaces for $\mathbf{X}$ are defined as \begin{align*}
    \Gamma_R(\mathbf{X}; m,\gamma,&k):=\{(X_1^{(k)},\dots,X_d^{(k)})\in (M_k(\C)_{\text{s.a.}})^d: \norm{X^{(k)}_j}\leq R\,\,\forall\,1\leq j\leq d,\text{ and } \\
    &\abs{\text{tr}_k(X_{i_1}X_{i_2}\cdots X_{i_p})-\tau(x_{i_1}x_{i_2}\cdots x_{i_p})}<\gamma\,\,\forall\,(i_1,\dots,i_m)\in [d]^p,\,\, 1\leq p\leq m \}.
\end{align*}

We identify isometrically $(M_k(\C))^d_{\text{s.a.}}$, equipped with the Hilbert-Schmidt metric, with $\R^{dk^2}$, equipped with the Euclidean metric, so that we can take the Lebesgue volume of the former space.

\begin{defn}[\cite{Voi93,Voi94}]
The \textit{microstates free entropy} of \textbf{X} is $$\chi(\mathbf{X})=\sup_{R>0}\inf_{m,\gamma}\limsup_{k\to\infty}\left[\frac{1}{k^2}\log\text{vol}_{dk^2}(\Gamma_R(\mathbf{X};m,\gamma,k)) + \frac{1}{2}d\log k\right].$$    
\end{defn}

In fact, the supremum over $R$ is superfluous, and one may evaluate the entropy for any fixed choice of $R$ which exceeds $\max_{1\leq j\leq d}\{\norm{x_j}_{\text{op}}\}$.

We now record the definition of conditional microstates free entropy as found in \cite{Shl02} (which is different from Voiculescu's original definition in \cite{Voi94}). 

Let $\mathbf{Y}=(y_1,\dots,y_r)$ be a self-adjoint $r$-tuple in $M$ with $\max_{1\leq j\leq r}\{\norm{y_j}_{\text{op}}\}<S$. We say that $\mathbf{Y}^{(k)}=(Y_1^{(k)},\dots,Y_r^{(k)})_{k\geq1}$ is a \textit{microstate sequence for Y} if $\mathbf{Y}^{(k)}\in (M_k(\C)_\text{s.a.})^r$ and, for each $n\in\N,\delta>0$, there exists an $l\in\N$ so that $\mathbf{Y}^{(l)}\in\Gamma_S(\mathbf{Y};n,l,\delta)$, and write $\mathbf{Y}^{(k)}\rightsquigarrow\mathbf{Y}$.

With $\mathbf{X,Y}$ as above, assume without loss of generality that $R>S$. Fix a microstate sequence $(\mathbf{Y}^{(k)})_{k\geq1}$ for $\mathbf{Y}$ and define \begin{align*}
    \Gamma_R(&\mathbf{X}\mid \mathbf{Y}^{(k)}\rightsquigarrow \mathbf{Y};m,k,\gamma):= \\&\{(X_1^{(k)},\dots,X_d^{(k)})\in(M_k(\C)_{\text{s.a.}})^d: (X_1^{(k)},\dots,X_d^{(k)},\mathbf{Y}^{(k)})\in \Gamma_R(\mathbf{X,Y};m,k,\gamma)\}.
\end{align*} 

Informally, these are those microstates $(X_1^{(k)},\dots,X_d^{(k)})$ for which $(X_1^{(k)},\dots,X_d^{(k)},\mathbf{Y}^{(k)})$ is a microstate for $\{\mathbf{X,Y}\}$. Accordingly, we make the following:

\begin{defn}[\cite{Shl02,JP24}]
    The \textit{microstates free entropy of} $\mathbf{X}$ \textit{relative to} $\mathbf{Y^{(k)}\rightsquigarrow Y}$ is 
    $$\chi(\mathbf{X}\mid\mathbf{ Y^{(k)}\rightsquigarrow Y})=\sup_{R>0}\inf_{m,\gamma}\limsup_{k\to\infty}\left[\frac{1}{k^2}\log\text{vol}_{dk^2}(\Gamma_R(\mathbf{X\mid Y^{(k)}\rightsquigarrow Y};m,\gamma,k) + \frac{1}{2}d\log k\right].$$
\end{defn}

Again, it is not necessary to take the supremum over the cutoff constant $R$.

We need to relate non-microstates and microstates free entropies. It was proven in \cite{BCG03} that $\chi\leq\chi^*$, and a conditional variant was obtained in \cite{JP24}. We will use the latter result, the statement of which we have simplified to suit our purposes.
\begin{thm}[\cite{BCG03,JP24}]
    Given finite self-adjoint tuples $\mathbf{X,Y}$ in a tracial von Neumann algebra $M$, assume that there is a microstate sequence $\mathbf{Y}^{(k)}\rightsquigarrow\mathbf{Y}$. Then $$\chi(\mathbf{X}\mid\mathbf{Y^{(k)}\rightsquigarrow Y})\leq \chi^*(\mathbf{X}:W^*(\mathbf{Y})).$$ 
\end{thm}
\subsection{Strong 1-boundedness and 1-bounded entropy}

Strong 1-boundedness was defined by Jung in \cite{Jun07}. We remark that Hadwin and Shen contemporaneously defined the closely related notion of free orbit dimension in \cite{HS07}. We employ the notation and terminology of \cite{Hay18}, which made explicit the entropic quantity implicit in \cite{Jun07}.

Given a set $S\subseteq\R^n$, denote by $\mathcal{K}_\epsilon(S)$ the minimal number of $\epsilon$-balls (with respect to the Euclidean metric) needed to cover $S$. Define

$$\mathbb{K}_\epsilon(\mathbf{X}):=\sup_{R>0}\inf_{m,\gamma}\limsup_{k\to\infty}\left[\frac{1}{k^2}\log\mathcal{K}_\epsilon(\Gamma_R(\mathbf{X};m,\gamma,k))\right],$$ making the analogous definition for $\mathbb{K}_\epsilon(\mathbf{X \mid Y^{(k)}\rightsquigarrow Y})$.

Suppose that there is $Y\in M_{\text{s.a.}}$ with diffuse spectrum and $\chi(Y)>-\infty$ so that $\mathbf{X}\cup\{Y\}$ is a finite generating set for $(M,\tau)$. In particular, there is a microstates sequence $\mathbf{Y}^{(k)}\to Y$. Jung showed in \cite{Jun07} that if $\mathbf{X}'$ is any other generating set of self-adjoints for $M$, then $$\sup_{\epsilon>0}\mathbb{K}_\epsilon(\mathbf{X}\cup\{Y\} \mid \mathbf{Y^{(k)}}\rightsquigarrow Y)=\sup_{\epsilon>0}\mathbb{K}_\epsilon(\mathbf{X' \mid Y^{(k)}}\rightsquigarrow Y).$$ This motivates the following
\begin{defn}[\cite{Hay18}, see also \cite{Jun07}]
    The \textit{1-bounded entropy} of $(M,\tau)$ is $$h(M):=\sup_{\epsilon>0}\mathbb{K}_\epsilon(\mathbf{X \mid Y^{(k)}}\rightsquigarrow Y)$$ for $\mathbf{X},Y$ as in the paragraph above.
\end{defn}
We will use the following property of 1-bounded entropy (see \cite{Hay18}): if $N_1\leq N_2\leq M_2\leq M_1$ is diffuse, then $h(N_1:M_1)\leq h(N_2:M_2)$.

For any free ultrafilter $\omega\in\beta\N\setminus\N$, we have $h(M)=h(M:M^\omega)$.

A diffuse, tracial von Neumann algebra $(M,\tau)$ is \textbf{strongly 1-bounded} if and only if $h(M)<\infty$. We refer the reader to \cite{Hay18,HJKE25,HJNS21} for further details and applications. If $\mathbf{X}$ generates a diffuse von Neumann algebra with $\delta_0(\mathbf{X})>1$, then $h(W^*(\mathbf{X}))=\infty$ (\cite{Hay18},\cite{Voi96}).

We now state the theorem linking Theorems 2.1 and 2.2 to strong 1-boundedness. In the examples above, when $N$ was diffuse and abelian, we took it to be separable in order to identify it with a standard probability space. In this section, we take $N$ to be separable as well; in the abelian case, this is equivalent to $N$ being singly generated, and is used throughout the proofs and definitions involving 1-bounded entropy.

\begin{lem}
    If $\mathbf{X}$ and $N=W^*(\mathbf{Y})$ satisfy (\ref{maineq}) with $N$ diffuse and hyperfinite, then $h(W^*(\mathbf{X}))<\infty$.
\end{lem}
\begin{proof}
    Fix a microstates sequence $\mathbf{Y}^{(k)}\rightsquigarrow\mathbf{Y}$. The sets $\Gamma_R(\mathbf{X+\sqrt{\epsilon}S:S};m,k,\gamma)$ will denote microstate spaces for $\mathbf{X+\sqrt{\epsilon}S}$ in the presence of $\mathbf{S}$; see \cite{Voi96}. Now define 
    \begin{align*}    \Gamma_R(\mathbf{X+\sqrt{\epsilon}S}\mid \mathbf{Y^{(k)}\rightsquigarrow Y}:\mathbf{S}&;m,k,\gamma):= \\
    &\Gamma_R(\mathbf{X+\sqrt{\epsilon}S\mid \mathbf{Y^{(k)}}\rightsquigarrow Y};m,k,\gamma)\cap\Gamma_R(\mathbf{X+\sqrt{\epsilon}S:S};m,k,\gamma),
    \end{align*}
    and the auxiliary quantity $\chi(\mathbf{X+\sqrt{\epsilon}S\mid\mathbf{Y^{(k)}}\rightsquigarrow Y}:\mathbf{S})$ the same way as all other forms of microstates entropy discussed at the beginning of this section. Clearly, \begin{align*}
        \limsup_{n\to\infty}\frac{1}{n^2}\log\text{vol}&\left( \Gamma_R(\mathbf{X+\sqrt{\epsilon}S\mid Y^{(k)}\rightsquigarrow Y}:\mathbf{S};m,k,\gamma)\right) \\ &\leq \limsup_{n\to\infty}\frac{1}{n^2}\log\text{vol}\left( \Gamma_R(\mathbf{X+\sqrt{\epsilon}S\mid Y^{(k)}\rightsquigarrow Y};m,k,\gamma)\right)
    \end{align*} which, in conjunction with Theorem 2.1 and Corollary 2.3, implies that $$\chi(\mathbf{X+\sqrt{\epsilon}S\mid\mathbf{Y^{(k)}}\rightsquigarrow Y}:\mathbf{S})\leq\chi(\mathbf{X+\sqrt{\epsilon}S\mid\mathbf{Y^{(k)}}\rightsquigarrow Y})$$ $$\leq\chi^*(\mathbf{X+\sqrt{\epsilon}S}:N)\leq\frac{d}{2}[\log(2\pi e)+\log\epsilon].$$
    
    On the other hand, repeating verbatim the (first half of the) proof of Lemma 2.2 in \cite{Jun03} with the set $\Gamma_R(\mathbf{X}\mid \mathbf{Y^{(k)}\rightsquigarrow Y};m,k,\gamma)$ replacing $\Gamma_R(\mathbf{X};m,k,\gamma)$ shows that $$\underline{\chi}(\mathbf{S})+\frac{d}{2}\log\epsilon+\mathbb{K}_{\epsilon'}(\mathbf{X}\mid \mathbf{Y^{(k)}\rightsquigarrow Y})\leq \chi(\mathbf{X}+\sqrt{\epsilon}\mathbf{S}\mid\mathbf{Y^{(k)}\rightsquigarrow Y}:\mathbf{S})\leq \frac{d}{2}\log\epsilon + \frac{d}{2}\log(2\pi e)$$ Here $\epsilon'=\Theta(\sqrt{\epsilon})$, with the proportionality constants depending only on $\mathbf{X,Y}$, while $\underline{\chi}$ is the microstates free entropy with $\liminf_k$ replacing the $\limsup_k$ in its definition. While in general it is not known when $\underline{\chi}=\chi$, it was shown in \cite{Voi98-2} that $\underline{\chi}(\mathbf{S})=\chi(\mathbf{S})$; in our case, they are both equal to $\frac{d}{2}\log(2\pi e)$. It follows that $\mathbb{K}_{\epsilon'}(\mathbf{X\mid Y^{(k)}\rightsquigarrow Y})\leq0$.

    Now, the definition of 1-bounded entropy requires that we fix microstates relative to a single self-adjoint, as opposed to the more general situation of a hyperfinite subalgebra, but the proof of lemmata A.4 and A.5 in \cite{Hay18} (really, the fact that all embeddings of a hyperfinite von Neumann algebra into $\mathcal{R}^\omega$ are unitarily conjugate; see \cite{JungHyperfinite}) shows that $h(W^*(\mathbf{X},\mathbf{Y}))\leq0$.
     Then, we compute $h(W^*(\mathbf{X}))=h(W^*(\mathbf{X}):W^*(\mathbf{X})^\omega)\leq h(W^*(\mathbf{X}):W^*(\mathbf{X,Y}))\leq h(W^*(\mathbf{X,Y}))\leq0.$
\end{proof}

If $\delta_0(\mathbf{X})>1$, $h(W^*(\mathbf{X}))=\infty$. Thus we have in particular the following
\begin{cor}
    There does not exist $N$ as in Theorem 2.1 for any set of generators of $L(\mathbb{F}_r)$ $(r>1)$, or a free product of diffuse, Connes-embeddable von Neumann algebras, so that (\ref{maineq}) would hold.
\end{cor}


\printbibliography
\end{document}